\pdfoutput=1
\documentclass[a4paper,11pt]{article}
\usepackage[english]{babel}
\usepackage{amsfonts}
\usepackage{amsthm}
\usepackage{amsmath}
\usepackage{amssymb}
\usepackage{enumerate}
\usepackage{mathtools}
\usepackage{tikz}
\usepackage{stmaryrd}
\expandafter\def\csname opt@stmaryrd.sty\endcsname
{only,shortleftarrow,shortrightarrow}
\usepackage{tikz-cd}
\usepackage[all]{xy}
\usepackage{graphicx}
\usepackage{amsmath}
\usepackage{adjustbox}
\usepackage{graphicx}
\usepackage{extpfeil}
\usepackage{comment}
\usepackage{mathabx}
\usepackage{float}
\usepackage[labelformat=empty]{caption}
\usepackage[toc]{appendix}
\usepackage[left=3.2cm,top=3cm,right=3.2cm,bottom=3cm]{geometry}
\usepackage{resizegather}
\usepackage[activate={true, nocompatibility}, final, tracking=true, kerning=true, spacing=true, factor=1100, stretch=10, shrink=10]{microtype}
\usepackage{dsfont}
\usetikzlibrary{matrix}
\usepackage{hyperref}
\hypersetup{
colorlinks=true,
linkcolor=[RGB]{0,0,204},
filecolor=magenta,      
urlcolor=cyan,
citecolor=[RGB]{0,204,0},
}
\usepackage[nameinlink]{cleveref}
\usepackage[normalem]{ulem}
\usepackage{sidecap}
\usepackage{mathrsfs}
\usepackage{mathscinet}

\usepackage{color} %CLASHES with hyperref
\definecolor{darkgreen}{rgb}{0.0, 0.7, 0.0}
\newenvironment{??}{\noindent \color{darkgreen}{\bf ???:} \footnotesize}{}
\definecolor{cyan}{cmyk}{1,0,0,0}

\newcommand{\bdg}{\begin{dg}}

\theoremstyle{plain}
\newtheorem{thm}{Theorem}[section]
\newtheorem{coroll}[thm]{Corollary}
\newtheorem{defn}[thm]{Definition}
\newtheorem{lemma}[thm]{Lemma}

\newtheorem{prop}[thm]{Proposition}

\newtheorem*{thm*}{Main Theorem}
\newtheorem{notn}[thm]{Notation}

\theoremstyle{definition}
\newtheorem{context}[thm]{Context}
\newtheorem{remark}[thm]{Remark}

\tikzset{
  symbol/.style={
    draw=none,
    every to/.append style={
      edge node={node [sloped, allow upside down, auto=false]{$#1$}}}
  }
}
\microtypecontext{spacing=nonfrench}

\DeclareMathOperator{\ad}{ad}

\newcommand{\ql}{\overline{\mathbb Q}_{\ell}}

\DeclareMathOperator{\ev}{ev}
\DeclareMathOperator{\Sh}{Sh}

\newcommand{\cM}{\mathcal{M}}

\newcommand{\cX}{\mathcal{X}}

\newcommand{\cZ}{\mathcal{Z}}
\newcommand{\cL}{\mathcal{L}}

\newcommand{\bA}{\mathbb{A}}
\newcommand{\bC}{\mathbb{C}}
\newcommand{\bG}{\mathbb{G}}

\newcommand{\bQ}{\mathbb{Q}}

\newcommand{\on}{\operatorname}

\newcommand{\Spec}{\on{Spec}}

\newcommand{\Higgs}{\mathscr{H}\!it}

\newcommand{\MHiggs}{\mathbb{M}_{Dol,G}}

\newcommand{\Mderham}{\mathbb{M}_{dR, G}}
\newcommand{\NHiggs}{\mathbb{N}_{Dol,G}}
\newcommand{\Nderham}{\mathbb{N}_{dR,G}}
\newcommand{\Bun}{ \mathscr{B}un}

\newcommand{\MHit}{\mathbb{M}_{Hit, G, \cL}}
\DeclareMathOperator{\Hcoh}{H}

\newcommand{\oql}{\overline{\mathbb{Q}}_{\ell}}

\usepackage[normalem]{ulem}
\usepackage{pbox}

\begin{document}
\title{\textbf{Topology of $\mathbb{G}_m$-actions and applications to the moduli of Higgs bundles}}
\author{Andres Fernandez Herrero and Siqing Zhang}
\date{}
\maketitle
\begin{abstract}
We explain some results concerning the topology of varieties and stacks equipped with an action of the multiplicative group $\mathbb{G}_m$. We apply these techniques to the moduli of Higgs bundles. Our main application is to upgrade the cohomological Nonabelian Hodge Theorem in positive characteristic \cite{dCGZ, herrero-zhang-naht} to an isomorphism of cohomology rings compatible with cup product.
\end{abstract}

\begin{section}{Introduction}

Let $C$ be a smooth projective connected curve of genus $g \geq 2$ over an algebraically closed field $k$ of positive characteristic $p>0$. Given a connected reductive group $G$ over $k$, the Nonabelian Hodge Theorem in \cite{groechenig-moduli-flat-connections, chen-zhu} relates the de Rham moduli stack $\cM_{dR,G}(C)$ of $G$-bundles with connection on the curve $C$ and the Dolbeault moduli stack $\cM_{Dol,G}(C')$ of $G$-Higgs bundles on the Frobenius twist $C'$. In \cite{dCGZ, herrero-zhang-naht}, this is refined to be compatible with semistability. In particular, for any given fixed degree $d \in \pi_1(G)$, there is a version of the Nonabelian Hodge Theorem relating the de Rham moduli space $\Mderham^{pd}(C)$ of semistable $G$-connections of degree $pd$ on $C$ and the Dolbeault moduli space $\MHiggs^d(C')$ of semistable $G$-Higgs bundles of degree $d$ on $C'$. 

The Dolbeault (resp. de Rham) moduli space admits a Hitchin (resp. de Rham-Hitchin) morphism to the Hitchin base $A_{G,\omega_{C'}}$.
Applying the Decomposition Theorem \cite{bbdg} to the (de Rham-)Hitchin morphism, we obtain an isomorphism between the intersection cohomology groups of the Dolbeault and de Rham moduli spaces \cite{dCGZ, herrero-zhang-naht}. 
However, for the ordinary cohomology rings of these moduli spaces, the Decomposition Theorem argument no longer applies, and \textit{loc.cit.} only have results locally over $A_{G,\omega_{C'}}$, relating certain cohomology sheaves.
Our main contribution in this paper is to complete the study of the Nonabelian Hodge Theorem at the level of cohomology rings by establishing the following global result.
\begin{thm*}[= \Cref{thm: main thm}]
    Let $k$ be an algebraically closed field of characteristic $p>0$. Let $G$ be a connected reductive group over $k$ that satisfies the low height property (\cite[Defn. 2.29]{herrero2023meromorphic}). Let $C$ be a smooth projective connected curve of genus $g \geq 2$, and let $C'$ denote the Frobenius twist of $C$. Then, for any given degree $d \in \pi_1(G)$, there is a canonical isomorphism of $\ell$-adic cohomology rings: 
    \[\Hcoh^{\bullet}(\MHiggs^d(C'), \ql) \cong \Hcoh^{\bullet}(\Mderham^{pd}(C), \ql).\]
\end{thm*}

We use the proof of the main theorem as an opportunity to discuss some general results on the topology of stacks and schemes equipped with actions of the multiplicative group $\bG_m$ (see \Cref{section: general results}). We make further use of these results in three additional applications:
\begin{itemize}
\item (\Cref{thm: comm diag}). A study of the cohomology of the nilpotent cones in the Dolbeault and de Rham moduli spaces. In the case $G=GL_n$ with $n$ coprime to $pd$, our result subsumes and supplies a correct proof of \cite[Thm. 2.4]{cohomological-nah} (see \Cref{mistake}).
    \item (\Cref{prop: cohomology of Higgs = bun}+ \Cref{coroll: purity of coh of higgs}). A study of the cohomology of the moduli stack of $G$-Higgs bundles. 
    \item (\Cref{prop: very stable G-bundles}). A quick proof of the characterization of very stable $G$ bundles in \cite[Thm. 1.1]{pauly_peon_nieto}.
\end{itemize}

\noindent \textbf{Acknowledgements.} We would like to thank Mark de Cataldo, Andr\'es Ib\'a\~nez N\'u\~nez, Weite Pi and Alfonso Zamora for useful conversations. We would also like to thank Tasuki Kinjo for very helpful comments on a previous version of this manuscript. This material is based
upon work supported by the National Science Foundation under Grant No. DMS-1926686.

    \begin{subsection}{Notation and conventions on derived categories of sheaves} 
    In this paper, we work over an algebraically closed ground field $k$. All stacks and schemes are assumed to live over $k$. We fix once and for all a prime number $\ell$ distinct from the characteristic of our ground field $k$. For any quasi-separated algebraic stack $\cX$ locally of finite type over $k$, we use the following:
    \begin{notn}[Derived category of $\ell$-adic sheaves] Let $D^b_c(\cX, \ql)$ denote the bounded derived category of constructible $\ql$-complexes on $\cX$ (see \cite[Rem. 3.21]{laszlo-olsson-adic-sheaves}). We denote by $\Sh_c(\cX, \ql)$ the abelian category of constructible $\ql$-sheaves on $\cX$; this is the heart of the standard $t$-structure on $D^b_c(\cX, \ql)$.
    \end{notn}

    Given an object in $D^{b}_c(\cX, \ql)$, we denote by $R\Gamma(\cX, \ql)$ the complex of derived global sections, which is an object in the bounded-below derived category $D^{+}(\ql)$ of $\ql$-vector spaces.

 If the ground field $k$ is $\bC$, then we use the following:
\begin{notn}[Derived category of mixed Hodge modules] Let $D_{H}^{+}(\cX)$ denote the bounded-below derived category of mixed Hodge modules constructed in \cite{tubach_mixed_hodge}. We denote by $D_{H,c}^{+}(\cX)$ the bounded-below derived category of cohomologically constructible mixed Hodge modules.
\end{notn} 
     
Pushing forward under the structure morphism $\cX \to \Spec(\bC)$ preserves constructibility of mixed Hodge modules. For any element $E \in D^{+}_{H,c}(\cX)$, we denote by $R\Gamma(\cX, E)$ the corresponding pushforward in $D^{+}_{H,c}(\Spec(\bC))$. There is a notion of weights for objects of $D_{H,c}(\Spec(\bC))$ (cf. \cite[Defn. 3.17]{tubach_mixed_hodge}), and therefore there is a notion of purity for complexes. We say that a complex in $D_{H,c}(\Spec(\bC))$ is pure if it is pure of weight $0$.

\end{subsection}
\end{section}

\begin{section}{Some general results} \label{section: general results}

\begin{subsection}{Cohomology of equivariant sheaves}
In this subsection, we work in the following context.
\begin{context} \label{context: A^1 retraction}
    Let $i : \cZ \hookrightarrow \cX$ be a closed immersion of quasi-separated algebraic stacks that are locally of finite type over $k$. Suppose that we have a morphism $H: \bA^1 \times \cX \to \cX$ such that $H|_{1 \times \cX}: \cX \to \cX$ is the identity and both restrictions $H|_{0 \times \cX}: \cX \to \cX$ and $H|_{\bA^1 \times \cZ}: \bA^1 \times \cZ \to \cX$ factor through $\cZ \subset \cX$.
\end{context}

        \begin{prop} \label{prop: A^1 retraction l-adic}
         In \Cref{context: A^1 retraction}, let $E \in D^b_c(\cX, \ql)$ be a complex of sheaves satisfying $H^*E \cong p_2^*E$, where $p_2: \bA^1 \times \cX \to \cX$ denotes the second projection. Then, the natural morphism of derived global sections $i^*: R\Gamma(\cX, E) \to R\Gamma(\cZ, i^*E)$ is an isomorphism in the derived category of $\ql$-vector spaces.
    \end{prop}
    \begin{proof}
        This is known in the context of varieties (see \cite[Lem. 4.2]{de2018combinatorics}), and the same proof applies verbatim in the context of algebraic stacks.
    \end{proof}

    \begin{prop} \label{coroll: vanishing cohomology equivariant sheaves}
        Let $A$ be a separated scheme locally of finite type over $k$ equipped with an action of $\bG_m$. Suppose that the action morphism $a: \bG_m \times A \to A$ extends to a morphism $\widetilde{a}: \bA^1 \times A \to A$. Let $i: A^{\bG_m} \hookrightarrow A$ denote the closed immersion from the subscheme of fixed points. Then, given any $\bG_m$-equivariant constructible sheaf $F$ in $\Sh_c(A, \ql)$, we have a canonical identification $R\Gamma(A, F) \cong R\Gamma(A^{\mathbb{G}_m}, i^*F)$. In particular, if $A^{\bG_m}$ has dimension $0$, then $\Hcoh^j(A, F) =0$ for all $j >0$.
    \end{prop}
    \begin{proof}
        The isomorphism $R\Gamma(A,F) \cong R\Gamma(A^{\mathbb{G}_m}, i^*F)$ follows from \cite[Prop. A.1]{khan2023fourier}. If $A^{\bG_m}$ has dimension $0$, then for any $j>0$ we get $\Hcoh^j(A,F) = \Hcoh^j(A^{\bG_m}, i^*F) =0$, where the last equality follows from the vanishing theorems in \cite[Exp. X, \S4]{sga4}.
    \end{proof}

    Note that \Cref{coroll: vanishing cohomology equivariant sheaves} does not follow from \Cref{prop: A^1 retraction l-adic} since the $\mathbb{G}_m$-equivariant sheaf $F$ may not satisfy $\widetilde{a}^*F\cong p_2^*F$.
    Indeed, in the example where $\widetilde{a}$ is the contracting action of $\bA^1$ on $\bA^n$, the condition $\widetilde{a}^*F\cong p_2^*F$ would force $F$ to be constant.

    If the ground field $k$ is $\bC$, then there is a version of \Cref{prop: A^1 retraction l-adic} in the context of mixed Hodge modules.
    \begin{prop} \label{prop: A^1 retraction mixed Hodge}
        In \Cref{context: A^1 retraction}, suppose that the ground field is $\mathbb{C}$, and let $E$ be an object in $D^{+}_{H,c}(\cX)$ satisfying $H^*E \cong p^*_2E$. Then, the natural morphism $i^*: R\Gamma(\cX, E) \to R\Gamma(\cZ, i^*E)$ is an isomorphism in $D_{H,c}^{+}(\Spec(\bC))$.
    \end{prop}
    \begin{proof}
        There is a forgetful functor $rat: D_{H,c}^{+}(\Spec(\bC)) \to D(\bQ)$ to the derived category of $\bQ$-vector spaces (where we think of the latter as the ind-completion of the bounded constructible derived category on the point). The functor $rat$ is conservative, and hence it suffices to check that $i^*: R\Gamma(\cX, E) \to R\Gamma(\cX, i^*E)$ is an isomorphism after applying $rat$. This can be proven by the same argument as in \Cref{prop: A^1 retraction l-adic}.
    \end{proof}

\end{subsection}

\begin{subsection}{Connectedness and fixed points}

\begin{defn}[{\cite[Defn. B1]{herrero2023meromorphic}}] \label{defn: contracting action}
Let $Y$ be a separated scheme locally of finite type over $k$. We say that an action of $\mathbb{G}_m$ on $Y$ is contracting if for any discrete valuation ring $R$ and morphism $f:Spec(R) \to Y$, there exists a $\mathbb{G}_m$-equivariant morphism $\widetilde{f}: \mathbb{A}^1_R \to Y$ such that $\widetilde{f}(1) = f$.
\end{defn}

Let $Y$ be separated scheme locally of finite type over $k$ which is equipped with a $\mathbb{G}_m$-action. Following \cite{drinfeld-attractor-loci}, we define the functor $Y^+$ from $k$-schemes to sets that sends $T \to S$ to the set of $\mathbb{G}_m$-equivariant morphism of $T$-schemes $\mathbb{A}^1_T \to Y_T$. The functor $Y^+$ is represented by an algebraic space locally of finite type over $k$ \cite[Prop. 1.4.1]{halpernleistner2018structure}. There are natural morphisms $\ev_0: Y^+ \to Y$ and $\ev_1: Y^+ \to Y$ defined by evaluating at $0\in \mathbb{A}^1$ and $1\in \mathbb{A}^1$ respectively. The morphism $\ev_1: Y^+ \to Y$ is a monomorphism of finite type, and so it follows that $Y^+$ is also a separated scheme locally of finite type over $k$ \cite[\href{https://stacks.math.columbia.edu/tag/03XX}{Tag 03XX}]{stacks-project}. The morphism $\ev_0: Y^+ \to Y$ factors through the closed subscheme $Y^{\mathbb{G}_m} \subset Y$ of $\mathbb{G}_m$-fixed points (denote by $Y^0$ in \cite[Prop. 1.4.1]{halpernleistner2018structure}). In terms of the morphism $\ev_1: Y^+ \to Y$, we have that the $\mathbb{G}_m$-action is contracting if and only if $\ev_1: Y^+ \to Y$ is a surjective closed immersion. In particular, if $Y$ is reduced and the $\mathbb{G}_m$-action is contracting, then $Y^+ = Y$.

\begin{prop}\label{prop: 2}
     Let $f: X \to Y$ be a proper $\mathbb{G}_m$-equivariant morphism of separated schemes that are locally of finite type $k$. Suppose that the $\mathbb{G}_m$-action on $Y$ is contracting. Then, $X$ is connected if and only if $X \times_{Y} Y^{\mathbb{G}_m}$ is connected.
\end{prop}
\begin{proof}
    By the discussion above, up to replacing $Y$ with its reduced subscheme, we may assume that the action $a: \bG_m \times Y \to Y$ extends to a morphism $H:\bA^1 \times Y \to Y$. 
    
    The scheme $X$ is connected if and only if we have $\Hcoh^0(X, \ql) = \ql$, and similarly $X \times_{Y} Y^{\bG_m}$ is connected if and only if $\Hcoh^0(X \times_{Y} Y^{\bG_m}, \ql) = \ql$. Therefore, it suffices to establish an isomorphism $\Hcoh^0(X, \ql) \cong \Hcoh^0(X \times_{Y} Y^{\bG_m}, \ql)$. Let $i: Y^{\bG_m} \hookrightarrow Y$ denote the closed immersion from the scheme of fixed points. By proper base change applied to the morphism $f$, we have $\Hcoh^0(X \times_{Y} Y^{\bG_m}, \ql) = \Hcoh^0(Y^{\bG_m}, i^* f_*\ql)$. Therefore, it suffices to show that the following is an isomorphism:
    \[\Hcoh^0(X, \ql) = \Hcoh^0(Y, f_*\ql) \xrightarrow{i^*} \Hcoh^0(Y^{\bG_m}, i^*f_*\ql) =\Hcoh^0(X \times_{Y} Y^{\bG_m}, \ql).\]
    This follows from \Cref{coroll: vanishing cohomology equivariant sheaves}, where we use $\cZ = Y^{\bG_m} \hookrightarrow Y = \cX$ and $F = f_*\ql$.
\end{proof}

The statement of \Cref{prop: 2} was an attempt to extract the topological content of the arguments in \cite{pauly_peon_nieto}. In particular, as a direct consequence of \Cref{prop: 2}, we may recover \cite[Thm. 1.1]{pauly_peon_nieto}. We state and prove a more general formulation of \cite[Thm. 1.1]{pauly_peon_nieto} in \Cref{subsection: very stable bundles}. 

\end{subsection}

\end{section}

\begin{section}{Applications to Non Abelian Hodge theory in positive characteristic} \label{section 3: nah positive char}

\begin{subsection}{Non Abelian Hodge Theory in positive characteristic}\label{subsec: non abelian hodge}

For this subsection, we fix the following setup.
\begin{context} \label{context: nah application} \quad 
\begin{enumerate}[(1)]
    \item Suppose that the algebraically closed ground field $k$ has characteristic $p>0$.
    \item Let $G$ be a connected reductive group over $k$ which satisfies the low height property (\cite[Defn. 2.29]{herrero2023meromorphic}).
    \item  Let $C$ be a smooth projective connected curve of genus $g \geq 2$, and let $C'$ denote the Frobenius twist of $C$.
    \item We fix a degree $d \in \pi_1(G)$.
\end{enumerate}  
\end{context}

The main varieties of interest in this subsection are the following. We refer the reader to \cite[\S2]{herrero2023meromorphic}
for detailed treatments of their constructions.
\begin{notn}[Moduli spaces]
    We denote by $\MHiggs^d(C')$ the Dolbeault moduli space of semistable $G$-Higgs bundles of degree $d$ on the Frobenius twist $C'$. 

    By a $G$-connection on $C$ we mean
    a $G$-bundle on $C$ together with a connection.
    We denote by $\Mderham^{pd}(C)$ the de Rham moduli space of semistable $G$-connections of degree $pd$ on $C$. 
\end{notn} 

We will use the corresponding Hitchin fibrations for both moduli spaces. See \cite[\S5]{herrero2023meromorphic} for more details on the construction of the following morphisms.
\begin{notn}[Hitchin morphisms]
    Let $A_{G, \omega_{C'}}$ denote the Hitchin base for the curve $C'$. We denote by $h_{Dol}: \MHiggs^d(C') \to A_{G, \omega_{C'}}$ (resp. $h_{dR}: \Mderham^{pd}(C) \to A_{G, \omega_{C'}}$) the Hitchin morphisms (resp. the de Rham-Hitchin morphism).
\end{notn}

There is a relative cup product for the morphism $h_{Dol}$ that equips $\bigoplus_{i =0}^{\infty} R^i(h_{Dol})_*\ql$ with the structure of a sheaf of graded commutative $\ql$-algebras on $A_{G, \omega_{C'}}$. Similarly, $\bigoplus_{i=0}^{\infty} R^i(h_{dR})_*\ql$ is naturally equipped with the structure of a sheaf of graded commutative $\ql$-algebras.

\begin{lemma} \label{lemma: same relative cohomology}
    With assumptions as in \Cref{context: nah application}, there is a canonical isomorphism of sheaves of graded commutative $\ql$-algebras
    \[can: \bigoplus_{i =0}^{\infty} R^i(h_{Dol})_*\ql \xrightarrow{\sim} \bigoplus_{i =0}^{\infty} R^i(h_{dR})_*\ql.\]
\end{lemma}
\begin{proof}
    We use the canonical isomorphism $can: \bigoplus_{i =0}^{\infty} R^i(h_{Dol})_*\ql \xrightarrow{\sim} \bigoplus_{i =0}^{\infty} R^i(h_{dR})_*\ql$ of graded $\ql$-sheaves constructed in \cite[Thm. 4.15(4)]{herrero-zhang-naht}. We need to check that $can$ is compatible with the algebra structures. Since the formation of the sheaves of graded commutative algebras is \'etale local on the base $A_{G, \omega_{C'}}$, we may check the compatibility of $can$ \'etale locally on $A_{G,\omega_{C'}}$. But, after passing to an \'etale cover of the base, the morphism $can$ is induced, by construction, from an isomorphism of the moduli spaces, and therefore it is compatible with the algebra structures induced by cup products.
\end{proof}

\begin{thm} \label{thm: main thm}
    With assumptions as in \Cref{context: nah application}, there is a canonical isomorphism of cohomology rings: 
    \begin{equation}
    \label{eq: main thm}
        \Hcoh^{\bullet}(\MHiggs^d(C'), \ql) \cong \Hcoh^{\bullet}(\Mderham^{pd}(C), \ql).
    \end{equation}
\end{thm}
\begin{proof}
    The Hitchin morphism $h_{Dol}: \MHiggs^d(C') \to A_{G, \omega_{C'}}$ is equivariant with respect to certain scaling $\mathbb{G}_m$-actions such that the action on $A_{G, \omega_{C'}}$ is contracting and has a unique fixed point $\mathfrak{o} \in A_{G, \omega_{C'}}(k)$. For every $i \geq 0$, the higher direct image $R^i(h_{Dol})_*\ql$ is a constructible $\bG_m$-equivariant $\ql$-sheaf on $A_{G, \omega_{C'}}$. Therefore, by \Cref{coroll: vanishing cohomology equivariant sheaves}, we have $\Hcoh^j(A_{G, \omega_{C'}}, R^i(h_{Dol})_*\ql) =0$ for all $j>0$. This implies that the $E_2$-page of the Leray spectral sequence for the composition $\MHiggs^d(C') \to A_{G, \omega_{C'}} \to \Spec(k)$ is concentrated in a single column, and hence we get a canonical isomorphism of $\ql$-algebras
    \begin{equation}
    \label{eq: iso on dol coh}
        \Hcoh^{\bullet}(\MHiggs^d(C'), \ql) \cong \bigoplus_{i =0}^{\infty} \Hcoh^{0}(A_{G,\omega_{C'}}, R^i(h_{Dol})_*\ql),
    \end{equation} 
    where the algebra structure on the right-hand side is induced from the algebra structure on the sheaf $\bigoplus_{i =0}^{\infty} R^i(h_{Dol})_*\ql$.

    Since $R^i(h_{Dol})_*\ql \cong R^i(h_{dR})_*\ql$ for all $i \geq 0$, we also have $\Hcoh^j(A_{G, \omega_{C'}}, R^i(h_{dR})_*\ql) =0$ for all $j>0$. Therefore, we have a similar canonical isomorphism of $\ql$-algebras induced by the Leray spectral sequence
    \begin{equation}
    \label{eq: iso on dr coh}
        \Hcoh^{\bullet}(\Mderham^{pd}(C), \ql) \cong \bigoplus_{i =0}^{\infty} \Hcoh^{0}(A_{G,\omega_{C'}}, R^i(h_{dR})_*\ql).
    \end{equation}
    Using the canonical isomorphism $can$ of sheaves of algebras from \Cref{lemma: same relative cohomology}, we obtain our desired chain of canonical identifications of $\ql$-algebras:
    \begin{gather*}
        \Hcoh^{\bullet}(\MHiggs^d(C'), \ql) \cong \bigoplus_{i =0}^{\infty} \Hcoh^{0}\left(A_{G,\omega_{C'}}, R^i(h_{Dol})_*\ql\right) \cong \bigoplus_{i =0}^{\infty} \Hcoh^{0}\left(A_{G,\omega_{C'}}, R^i(h_{dR})_*\ql\right) \cong \Hcoh^{\bullet}(\Mderham^{pd}(C), \ql).
    \end{gather*}
\end{proof}
\end{subsection}

\begin{subsection}{The cohomology of the nilpotent cones}\label{sec: coh nilp}

We keep the notation in \Cref{subsec: non abelian hodge}.
\begin{notn}[Nilpotent cones]
Let $\mathfrak{o}$ be the origin of the Hitchin base $A_{G,\omega_{C'}}$.
    We denote by $\NHiggs^d(C')$ (resp. $\Nderham^{pd}(C)$) the fiber $h_{Dol}^{-1}(\mathfrak{o})$ (resp. $h_{dR}^{-1}(\mathfrak{o})$).
\end{notn}
The semistable Nonabelian Hodge correspondence in \cite[Thm. 4.10]{herrero-zhang-naht} is an $A_{G,\omega_{C'}}$-isomorphism of the form:
\begin{equation}
    \label{eq: semistable naht}\mathbb{H}^o\times^{\mathbb{P}^o}\MHiggs^d(C')\xrightarrow{\sim}\Mderham^{pd}(C),
\end{equation}
where $\mathbb{P}^o$ is a smooth group scheme over $A_{G,\omega_{C'}}$ with connected geometric fibers acting on $\MHiggs^d(C')$, and $\mathbb{H}^o$ is a $\mathbb{P}^o$-torsor over $A_{G,\omega_{C'}}$.

Any given trivialization of the fiber of the torsor $\mathbb{H}^o|_{\mathfrak{o}}$ induces an isomorphism of schemes $\NHiggs^d(C')\xrightarrow{\sim}\Nderham^{pd}(C)$.
 (Note that \cite{ogus2007nonabelian} shows that any $W_2(k)$-lift of $C$ induces such a trivialization).
Since the group scheme $\mathbb{P}^o|_{\mathfrak{o}}$ is connected, the Homotopy Lemma \cite[Lem. 4.13]{herrero-zhang-naht} entails that the induced isomorphism of cohomologies $\Hcoh^{\bullet}(\NHiggs^d(C'))\xrightarrow{\sim} \Hcoh^{\bullet}(\Nderham^{pd}(C))$ is independent of the choice of trivialization of $\mathbb{H}^o|_{\mathfrak{o}}$.
Therefore, the isomorphism \eqref{eq: semistable naht} induces a canonical identification of cohomology rings
\begin{equation}
    \label{eq: iso on ns}
    \Hcoh^{\bullet}(\NHiggs^d(C'), \ql)\xrightarrow{\sim} \Hcoh^{\bullet}(\Nderham^{pd}(C), \ql).
\end{equation}

\begin{thm}
\label{thm: comm diag}
    With assumptions as in \Cref{context: nah application}, we have the following commutative diagram of isomorphisms of cohomology rings:
    \begin{equation}
    \label{diag: iso cones}
        \xymatrix{
        \Hcoh^{\bullet}(\MHiggs^{d}(C'), \ql)\ar[r]^-{\eqref{eq: main thm}}_-{\cong} \ar[d]_-{\cong}
        & 
        \Hcoh^{\bullet}(\Mderham^{pd}(C),\ql) \ar[d]^-{\cong}\\
        \Hcoh^{\bullet}(\NHiggs^{d}(C'), \ql) \ar[r]_-{\eqref{eq: iso on ns}}^-{\cong} 
        &
        \Hcoh^{\bullet}(\Nderham^{pd}(C),\ql),
        }
    \end{equation}
    where both vertical morphisms are induced by the restrictions.
\end{thm}
\begin{proof}
     Applying \Cref{coroll: vanishing cohomology equivariant sheaves} to the $\mathbb{G}_m$-equivariant sheaf $R^i(h_{Dol})_*\ql$ on $A_{G,\omega_{C'}}$, we obtain a canonical isomorphism of $\ql$-algebras
     \begin{equation}
     \label{eq: restr on dol}
         \bigoplus_{i=0}^{\infty}\Hcoh^0(A_{G,\omega_{C'}},R^i(h_{Dol})_*\ql)\cong \bigoplus_{i=0}^{\infty} \Hcoh^0(\mathfrak{o}, i_{\mathfrak{o}}^*R^i(h_{Dol})_*\ql),
     \end{equation}
     where $i_{\mathfrak{o}}: \mathfrak{o}\hookrightarrow A_{G,\omega_{C'}}$ is the closed immersion of the origin.
Combining \eqref{eq: iso on dol coh}, \eqref{eq: restr on dol} and proper base change, we see that the left vertical restriction morphism in \eqref{diag: iso cones} is an isomorphism of cohomology rings.
    By \Cref{lemma: same relative cohomology}, \eqref{eq: restr on dol}, and \eqref{eq: iso on dr coh}, we have that the right vertical restriction morphism in \eqref{diag: iso cones} is also an isomorphism of cohomology rings.

    To show that \eqref{diag: iso cones} is commutative, it suffices to prove that the isomorphism \eqref{eq: iso on ns} is identified with the isomorphism $\Hcoh^0(i_{\mathfrak{o}}^*can)$, where the morphism $can$ is as in \Cref{lemma: same relative cohomology}.
    The morphism $can$ is obtained in the following way: we choose \'etale local trivializations of the $\mathbb{P}^o$-torsor $\mathbb{H}^o$ over $A_{G,\omega_{C'}}$, which induces \'etale local isomorphisms of the form $R^i(h_{Dol})_*\ql\xrightarrow{\sim} R^i(h_{dR})_*\ql$. 
    We then use the Homotopy Lemma \cite[Lem. 4.13]{herrero-zhang-naht} to show that those \'etale local isomorphisms are independent of the trivializations of $\mathbb{H}^o$, and thus glue to a global isomorphism $can$.
    Therefore, up to the proper base change isomorphism, the morphism $\Hcoh^0(i_{\mathfrak{o}}^*can)$ is given by an isomorphism of schemes $\NHiggs^d(C')\xrightarrow{\sim} \Nderham^{pd}(C)$ induced by \eqref{eq: semistable naht} and a trivialization of $\mathbb{H}^o|_{\mathfrak{o}}$.
    By the discussion above \Cref{thm: comm diag}, we see that $\Hcoh^0(i_{\mathfrak{o}}^*can)$ coincides with \eqref{eq: iso on ns} up to the proper base change isomorphism. 
\end{proof}

\begin{remark}\label{mistake}
    
 In the case when $G=\mathrm{GL}_n$ and the degree $d \in \pi_1(\mathrm{GL}_n) \cong \mathbb{Z}$ is coprime to $n$, the isomorphisms in \Cref{thm: comm diag} recovers \cite[Thm. 2.4]{cohomological-nah}, which is one of the main results in that paper. We take this opportunity to remark that there is an gap in the proof \textit{loc. cit}, and our method above supplies the correct proof:
Let us use the notation in \textit{loc. cit.}
The main theorem, which is called the Cohomological Simpson Correspondence, is proved twice: first in Theorem 2.1 and then in Theorem 3.6. 
The latter is a refinement of the former, with a different proof. 
The proof of Theorem 2.1 contains a gap: we cannot deduce (22), which says $h_{dR,*} \oql \cong h_{Hod,0,*}\oql$, from (17), which says that $h_{dR}\times \mathrm{id}_{\bG_m}= h_{Hod}|_{\bG_m}$, and (21), which identifies the two terms in (17) with the nearby cycles of $(h_{dR}\times\mathrm{id}_{\bG_m})_*\oql$ and $(h_{Hod}|_{\bG_m})_*\oql$.
The problem is that there is an implicit non-identity isomorphism on the common target of the two morphisms in (17), see \cite[Lem. 4.5]{de-Z-2022projective} for the precise statement.
We do not know if (22) is true.
Because of Theorem 3.6, the Cohomological Simpson Correspondence still holds, but the flawed proof of Theorem 2.1 affects the proof Theorem 2.4, where (22) is used.
Our result above shows that Theorem 2.4, and all the results depending on it, are still true. 

Let us also remark that the idea of using results similar to \cite[Lem. 4.2]{de2018combinatorics} to show the $E_2$-degeneration of the Leray spectral sequence as in the proof of \Cref{thm: main thm} is suggested by de Cataldo when we were trying to remedy \cite[Thm. 2.4]{cohomological-nah}. 
\end{remark}

\end{subsection}

\section{Further Applications}

In this section, we use our study of the topology of $\mathbb{G}_m$-actions in \Cref{section: general results} to obtain a result on the cohomology of the moduli stack of $G$-Higgs bundles and a result on very stable $G$-bundles.
We believe that these results are of independent interest.

\begin{subsection}{The cohomology of the stack of $G$-Hitchin pairs}

In this subsection, we study the cohomology of the whole moduli stack of $G$-Higgs bundles without imposing any form of semistability (as opposed to the case of the moduli spaces considered in \Cref{section 3: nah positive char}). We place ourselves in the following
\begin{context} \label{context: application coh bun = higgs }
    Fix a connected reductive group $G$ and a smooth connected projective curve $C$ over the algebraically closed ground field $k$. We fix a degree $d \in \pi_1(G)$ and a line bundle $\cL$.
\end{context}

\begin{notn}
    Let $\Bun_G^d(C)$ denote the stack that parametrizes $G$-bundles of degree $d$ on $C$ (see \cite[Thm. 5.8]{hoffmann-connected-components} for the notion of degree of a $G$-bundle). 
\end{notn}

\begin{defn}[Stack of $G$-Hitchin pairs]
    We denote by $\Higgs_{G,\cL}^d(C)$ the algebraic stack that parametrizes pairs $(E, \varphi)$ consisting of a $G$-bundle $E$ of degree $d$ on $C$ and an $\cL$-twisted Higgs field $\varphi \in \Hcoh^{0}(C, \ad(E) \otimes \cL)$, where $\ad(E) := E \times^G \text{Lie}(G)$ is the adjoint vector bundle.
\end{defn}

\begin{prop} \label{prop: cohomology of Higgs = bun}
    Suppose that we have the setup as in \Cref{context: application coh bun = higgs }. Then, we have a natural isomorphism $R \Gamma(\Higgs_{G,\cL}^d(C),\ql ) \to R\Gamma(\Bun_G^d(C), \ql)$. If the ground field is $\bC$, there is an isomorphism $R\Gamma(\Higgs_{G,\cL}^d(C), \bQ) \to R\Gamma(\Bun_G^d(C), \bQ)$ in $D^{+}_{H, c}(\Spec(\bC))$.
\end{prop}
\begin{proof}
    This follows from applying \Cref{prop: A^1 retraction l-adic} and \Cref{prop: A^1 retraction mixed Hodge} where we set $\cX = \Higgs_{G,\cL}^d$ and $\cZ = \Bun_G^d(C)$, and we use the morphisms
    \[ i: \Bun_G(C)^d \hookrightarrow \Higgs_{G,\cL}^d(C), \;\;\; E \mapsto (E, 0)\]
    \[ H: \bA^1 \times \Higgs_{G,\cL}^d(C) \to \Higgs_{G,\cL}^d(C), \; \; \; (t, (E, \varphi)) \mapsto (E, t\cdot \varphi).\]
\end{proof}

\begin{coroll} \label{coroll: purity of coh of higgs}
    In \Cref{context: application coh bun = higgs }, suppose that $G$ is semisimple and that the ground field is $\mathbb{C}$. Then, the cohomology $\Hcoh^{\bullet}(\Higgs_{G,\cL}^d, \bQ)$ is pure.
\end{coroll}
\begin{proof}
This is a consequence of \Cref{prop: cohomology of Higgs = bun} and the fact that the cohomology $\Hcoh^{\bullet}(\Bun_G^d(C), \bQ)$ is pure by \cite{atiyah-bott}, see also \cite[Cor. 4.5]{arapura2008motive}.
\end{proof}

\begin{remark}
    If the degree of the line bundle $\cL$ satisfies $\deg(\cL) >2g-2$, where $g$ is the genus of the curve, then it was recently observed that, more surprisingly, the Borel-Moore homology of $\Higgs_{G,\cL}^d$ is also pure (see \cite[Rem. 5.6]{kinjo2024decompositiontheoremgoodmoduli}).
\end{remark}

\begin{remark}
    If we use the notion of purity in the sense of Frobenius eigenvalues, then \Cref{coroll: purity of coh of higgs} also applies in the case when $C$ is defined over a finite field and $G$ is semisimple by the main results in \cite{heinloth-schmitt}.
\end{remark}

\end{subsection}

\begin{subsection}{Very stable $G$-bundles} \label{subsection: very stable bundles}
 For this subsection, we place ourselves in the following.

\begin{context} \label{context: very stable application}
        Fix a connected reductive group $G$ and a smooth connected projective curve $C$ over the algebraically closed ground field $k$. We fix a degree $d \in \pi_1(G)$ and a line bundle $\cL$. If the field $k$ has positive characteristic, then we impose the assumption that $G$ satisfies the low height property as in \cite[Def. 2.29]{herrero2023meromorphic}.
\end{context}

\begin{notn}
    Let $\MHit^d(C)$ denote the moduli space of semistable $\cL$-twisted $G$-Hitchin pairs on $C$ (cf. \cite[Lem. 5.18]{herrero2023meromorphic}).
\end{notn} 

The scheme $\MHit^d(C)$ is of finite type over $k$ equipped with a proper Hitchin morphism $h: \MHit^d(C) \to A_{G,\cL}$ into the corresponding affine Hitchin base $A_{G,\cL}$ \cite[Cor. 6.21+Rmk. 6.22]{ahlh}.
\begin{defn}
    A stable $G$-bundle $E$ is called very stable with respect to $\cL$ if and only if its only nilpotent $\cL$-twisted Higgs field $\varphi \in \Hcoh^{0}(C,\ad(E) \otimes \cL)$ is the zero section.
\end{defn}
Given a stable $G$-bundle $E$, let $V_G$ denote the vector space $\Hcoh^{0}(C, \ad(E) \otimes \cL)$ thought of as an affine scheme. There is a locally closed immersion
\[ V_G \hookrightarrow \MHit^d(C), \; \; \; \varphi \mapsto (E, \varphi).\]
    \begin{thm}[{\cite[Thm. 1.1]{pauly_peon_nieto}}] \label{prop: very stable G-bundles}
        Suppose that we are in the setup of \Cref{context: very stable application}. Given a stable $G$-bundle $E$, the following statements are equivalent:
        \begin{enumerate}[(a)]
        \item The inclusion $V_G \hookrightarrow \MHit^d(C)$ is a closed immersion.
            \item The composition $h_{V_G}: V_G \to \MHit^d(C) \to A_{G,\cL}$ is proper.
            \item The morphism $h_{V_G}: V_G \to A_{G,\cL}$ is quasifinite.
            \item $E$ is very stable with respect to $\cL$.
        \end{enumerate}
    \end{thm}
    \begin{proof}  \quad 
    
    (a) $\Leftrightarrow$ (b). This is immediate.
   
            (b) $\Rightarrow$ (c). This follows because the source $V_G$ is affine.
            
             (c) $\Rightarrow$ (b). We equip the vector space $V_G$ with the contracting scaling $\bG_m$-action, and we equip the Hitchin base $A_{G,\cL}$ with its standard contracting $\bG_m$-action. The origin $\mathfrak{o} \in A_{G,\cL}(k)$ is the unique $\bG_m$-fixed point in $A_{G,\cL}$. If $h_{V_G}$ is quasifinite, then $h^{-1}_{V_G}\left(A_{G,\cL}^{\bG_m}\right) = h^{-1}_{V_G}(\mathfrak{o})$ is finite. Hence \cite[Prop. B5]{herrero2023meromorphic} applies to show that  $h_{V_G}$ is proper.
                
            (c) $\Rightarrow$ (d). By the equivalences proven in the paragraphs above, we know that $h_{V_G}$ is finite. Equip $V_G$ and $A_{G,\cL}$ with the $\bG_m$-actions as explained in the proof of (c) $\Rightarrow$ (b). Since $V_G$ is connected, \Cref{prop: 2} implies that $h_{V_G}^{-1}(\mathfrak{o})$ is connected. Since $h_{V_G}^{-1}(\mathfrak{o})$ is finite, it must consist of a single point. By definition, the points in $h_{V_G}^{-1}(\mathfrak{o})$ correspond to nilpotent Higgs fields on $E$, and therefore we conclude that the only nilpotent Higgs field is $0$. Hence, $E$ is very stable.
            
            (d) $\Rightarrow$ (c). Equip $V_G$ and $A_{G,\cL}$ with the $\bG_m$-actions as explained in the proof of (c) $\Rightarrow$ (b). By upper-semicontinuity of fiber dimension and the existence of zero limits for the $\bG_m$-action on $V_G$, it follows that the dimension of the preimage $h^{-1}_{V_G}(\mathfrak{o})$ of the unique fixed point $\mathfrak{o} \in A_{G,\cL}(k)$ is maximal. By assumption, this preimage consists of a single point, and so it has dimension $0$. It follows that all the fibers of $h_{V_G}$ have dimension $0$, as desired.  
    \end{proof}

    \begin{remark}
        \Cref{prop: 2} can also be used to prove the natural generalization of \cite[Thm. 1.1]{hausel-hitchin} to the setting of $G$-Hitchin pairs.
    \end{remark}
\end{subsection}
    
\end{section}

\bibliographystyle{alpha}
\footnotesize{\bibliography{categorical.bib}}

  \textsc{Department of Mathematics, University of Pennsylvania,
    David Rittenhouse Laboratory, 209 S 33rd St, Philadelphia, PA 19104,
USA}\par\nopagebreak
  \textit{E-mail address}: \texttt{andresfh@sas.upenn.edu}

       \textsc{School of Mathematics, Institute for Advanced Study, 1 Einstein Drive, Princeton, NJ, 08540,
USA}\par\nopagebreak
  \textit{E-mail address}: \texttt{szhang@ias.edu}

\end{document}